\documentclass[12pt]{article}
\usepackage{graphicx}
\usepackage{amssymb}
\usepackage{amsmath}
\usepackage{amsfonts}
\usepackage{amsthm}
\usepackage[english]{babel}
\usepackage[latin1,ansinew]{inputenc}
\usepackage{a4wide}
\usepackage{color}
\newcommand{\be}{\begin{eqnarray}}
\newcommand{\ee}{\end{eqnarray}}
\newtheorem{theo}{Theorem}%[section]

\newtheorem{coro}{Corollary}
\newtheorem{defi}{Definition}
\newcommand{\CP}{{^C\!P}}
\newcommand{\DP}{{^D\!P}}
\newcommand{\SP}{{^S\!P}}
\newcommand{\BP}{{^B\!P}}
\newcommand{\QP}{{^Q\!P}}
\newcommand{\HP}{{^H\!P}}
\newcommand{\R}{\mathbb R}
\newcommand{\C}{\mathbb C}
\newcommand{\bH}{\mathbb H}

\newcommand{\mT}{{\mathcal T}}
\newcommand{\eps}{\epsilon}

\begin{document}
\date{June 30, 2021}

\title{A Galilei Invariant Version of Yong's Model} 

\author{Heinrich Freist\"uhler}
\maketitle
\vspace{-.5cm}
In his pioneering paper \cite{Y}, Yong has described the dynamics of a compressible fluid 
with Maxwell delayed viscosity as a symmetric hyperbolic system of balance laws, and shown 
that the solutions of this system tend to solutions of the Navier-Stokes equations.  

The original main purpose of the present note was to propose a Galilei invariant version of 
Yong's model, as a system of balance laws (while also communicating some simple observations on 
its shock waves). This has changed in the course of time.

\textit{The model.} 
For a barotropic fluid, $e=e(v), v=1/\rho,$ and $p=p(v)=-e'(v)$, with $-p'(v)=e''(v)>0$
consider the equations of motion
\be\label{YG}
\begin{aligned}
\rho_t+(\rho u^j)_{,j}&=0\\
(\rho u^i)_t +(\rho u^ju^i+p\delta^{ij})_{,j}+\frac1\epsilon(\sigma^{ij}+\tau\delta^{ij})_{,j}&=0\\
(\rho\sigma^{kl})_t+(\rho u^j\sigma^{kl})_{,j}
+\frac1\epsilon c^{klj}_{i}u^i_{,j}&=-\frac{\sigma^{kl}}{\mu\eps^2}\\
(\rho\tau)_t+(\rho u^j\tau)_{,j}
+\frac1\epsilon u^j_{,j}&=-\frac{\tau}{\nu\eps^2}
\end{aligned}
\ee
with
$$
%\frac12(u^{k,l}+u^{l,k})-\frac13 u^{r}_{,r}\delta^{kl}=C^{klj}_{i}u^i_{,j},
%\quad
c^{klj}_i=\frac12(\delta^k_i\delta^{lj}+\delta^l_i\delta^{kj})-\frac13 \delta^r_i\delta^j_r\delta^{kl},
$$
and define the total energy and its flux as
\be\label{EEF}
E=\rho(e+\frac12u_iu^i+\frac12\sigma_{kl}\sigma^{kl}+\frac12\tau^2),\quad 
F^j=((E+p)\delta^{ij}+\frac1\epsilon(\sigma^{ij}+\tau\delta^{ij}))u_i.
\ee
\begin{theo}
(i) Using $(E,F^j)$ from \eqref{EEF} as an (``entropy'', ``entropy'' flux) pair, 
system \eqref{YG} can be written as a symmetric hyperbolic system of balance laws.
All its solutions dissipate energy,
\be\label{divEG<0}
E_t+F^j_{,j}\le 0.
\ee
(ii) For any fixed values of $\epsilon >0$ and $\rho_*>0$ and any data 
$({^0\rho}_0,{^0u}^i,{^0\sigma}^{kl},{^0\tau})$ such that
$({^0\rho}_0-\rho_*,{^0u}^i,{^0\sigma}^{kl},{^0\tau})\in H^s(\R^d), s>d/2+1$
and ${^0\rho}_0>0$ is 
uniformly bounded away from $0$,
there exists a $\bar t>0$ such that system \eqref{YG} has a solution in $C^1([0,\bar t],H^s(\R^d))$ that 
assumes these data.
\goodbreak
(iii) If for certain initial data, the Navier-Stokes equations 
\be\label{NS}
\begin{aligned}
\rho_t+(\rho u^j)_{,j}&=0\\
(\rho u^i)_t +(\rho u^ju^i+p\delta^{ij})_{,j}&=
(\mu c^{ijl}_ku^k_{,l}+\nu\delta^{ij}u^k_{,k})_{,j}
\end{aligned}
\ee
have a solution in $C^1([0,\bar t],H^s(\R^d))$ for some $\bar t>0$,
then there exists an $\bar\epsilon>0$ such that for 
every $\epsilon\in(0,\bar\epsilon)$ system \eqref{YG} has a unique  
solution in $C^1([0,\bar t],H^s(\R^d))$ for the same data
with $\sigma^{kl}=\epsilon\mu c_i^{klj} u^i_{,j}, \tau=\epsilon\nu u^j_{,j}$, 
and the $(\rho,u)$-parts of these solutions converge, 
as $\epsilon\searrow 0$, to the solution of \eqref{NS}. 
\end{theo}
\begin{proof}
(i) Noting that 
$$
\begin{aligned}
 \rho(u^i_t+u^ju^i_{,j})+\delta^{ij}p_{,j}+\frac1\epsilon(\sigma^{ij}+\tau\delta^{ij})_{,j}&=0,\\
\rho(\sigma^{kl}_t+u^j\sigma^{kl}_{,j})
+\frac1\epsilon c^{klj}_{i}u^i_{,j}&=-\frac{\sigma^{kl}}{\mu\eps^2},\\
\rho(\tau_t+u^j\tau)_{,j}
+\frac1\epsilon u^j_{,j}&=-\frac{\tau}{\nu\eps^2},
\end{aligned}
$$
and
$$
\sigma_{kl}c_i^{klj}
=
\sigma_{kl}(\frac12(\delta^k_i\delta^{lj}+\delta^l_i\delta^{kj})-\frac13 \delta^r_i\delta^j_r\delta^{kl})
=\sigma_i^j,
$$ 
one finds
$$
\begin{aligned}
&E_t+F^j_{,j}
\\
&=\rho(e+\frac12 u_iu^i+\frac12\sigma_{kl}\sigma^{kl}+\tau^2)_t
+\rho u^j(e+\frac12 u_iu^i+\frac12\sigma_{kl}\sigma^{kl}+\tau^2)_{,j}+(pu^j)_{,j}
+\frac1\epsilon((\sigma^{ij}+\tau\delta^{ij})u_i)_{,j}\\
&=\frac p\rho(\rho_t+u^j\rho_{,j})+\rho u_i(u^i_t+u^ju^i_{,j})
+\rho\sigma_{kl}(\sigma^{kl}_t+u^j\sigma^{kl}_{,j})+\rho\tau(\tau_t+u^j\tau_{,j})
+(pu^j)_{,j}
+\frac1\epsilon((\sigma^{ij}+\tau\delta^{ij})u_i)_{,j}\\
&=\!-pu^j_{,j}\!\!-u_i(\delta^{ij}p_{,j}+
\frac1\epsilon(\sigma^{ij}\!+\!\tau\delta^{ij})_{,j})
-\sigma_{kl}(\frac1\epsilon c^{klj}_{i}u^i_{,j}\!+\!\frac{\sigma^{kl}}{\mu\eps^2})
-\tau(\frac1\epsilon u^j_{,j}\!+\!\frac{\tau}{\nu\eps^2})
+(pu^j)_{,j}+\frac1\epsilon((\sigma^{ij}\!+\!\tau\delta^{ij})u_i)_{,j}\\
&=-u_i(\frac1\epsilon\sigma^{ij}_{,j})
-\sigma_{kl}(\frac1\epsilon c^{klj}_{i}u^i_{,j})
+\frac1\epsilon(\sigma^{ij}u_i)_{,j}-\sigma_{kl}\frac{\sigma^{kl}}{\mu\eps^2}-u_i(\frac1\epsilon(\tau\delta^{ij})_{,j})
-\tau(\frac1\epsilon u^j_{,j}+\frac{\tau}{\nu\eps^2})
+\frac1\epsilon(\tau\delta^{ij}u_i)_{,j}\\
&=\frac1\epsilon(-u_i\sigma^{ij}_{,j}
-\sigma^j_iu^i_{,j}+(\sigma^{ij}u_i)_{,j})
-\sigma_{kl}\frac{\sigma^{kl}}{\mu\eps^2}+\frac1\epsilon(-u_i(\tau\delta^{ij})_{,j}-\tau u^j_{,j}+(\tau\delta^{ij}u_i)_{,j})
-\frac{\tau^2}{\nu\epsilon^2}\\
&=-\frac{\sigma_{kl}\sigma^{kl}}{\mu\eps^2}-\frac{\tau^2}{\nu\epsilon^2}.
\end{aligned}
$$
This proves \eqref{divEG<0}. 

For solutions to the homogeneous system associated with \eqref{YG}, the 
same calculation yields the corresponding identity 
\be\label{divEG=0}
E_t+F^j_{,j}=0.
\ee
Since $E$ as a function of the conserved quantities $\rho, m^j=\rho u^j, S^{kl}=\rho\sigma^{kl}, \mT=\rho\tau$,
\be
E=\rho e+\frac{m^2+S^2+\mT^2}{2\rho}\quad\text{with }\quad
m^2={m_im^i},\  S^2={S_{kl}S^{kl}},
\ee
is convex,  
$$
D^2E={\begin{pmatrix}
     \displaystyle{\frac{p'(\rho)}{\rho}+\frac{m^2+S^2+\mT^2}{\rho^3}}&
     \displaystyle{-\frac{m^j}{\rho^2}}&
     \displaystyle{-\frac{S^{rs}}{\rho^2}}
     \displaystyle{-\frac{\mT}{\rho^2}}&0\\
     \displaystyle{-\frac{m_i}{\rho^2}}&\displaystyle{\frac1\rho \delta_i^j}&0&0 \\
     \displaystyle{\frac{S_{kl}}{\rho^2}}&0&\displaystyle{\frac1\rho\delta_{kl}^{rs}}&0\\
     \displaystyle{-\frac{\mT}{\rho^2}}&0&0&\displaystyle{\frac1\rho}
\end{pmatrix}>0},
$$
the remaining part of the assertion follows from \cite{FL}.

Assertions (ii) and (iii) follow in direct analogy with Yong's beautiful argument for his Theorem 3.1  in \cite{Y}. 
\end{proof}

Very similar considerations hold for nonbarotropic fluids, 
$e=e(v,\varsigma)$ with 
\be\label{D2e}
D^2e(v,\varsigma)>0,\ p=-e_v(v,\varsigma),
\ee
the equations of motion now reading
\be\label{YGnb}
\begin{aligned}
\rho_t+(\rho u^j)_{,j}&=0\\
(\rho u^i)_t +(\rho u^ju^i+p\delta^{ij})_{,j}+\frac1\epsilon(\sigma^{ij}+\tau\delta^{ij})_{,j}&=0\\
E_t+F^j_{,j}&=0\\
(\rho\sigma^{kl})_t+(\rho u^j\sigma^{kl})_{,j}
+\frac1\epsilon c^{klj}_{i}u^i_{,j}&=-\frac{\sigma^{kl}}{\mu\eps^2}\\
(\rho\tau)_t+(\rho u^j\tau)_{,j}
+\frac1\epsilon u^j_{,j}&=-\frac{\tau}{\nu\eps^2}.
\end{aligned}
\ee
The same calculation as above, now taking into account that in this situation
$e$ and $p$ depend also on the specific entropy $\varsigma$, leads to the identity
\be\label{sup}
\rho\theta(\varsigma_t+u^j\varsigma_{,j})=
\frac{\sigma_{kl}\sigma^{kl}}{\mu\epsilon^2}+\frac{\tau^2}{\nu\epsilon^2}
\ee
where $\theta=e_s(v,\varsigma)$ is the temperature.
Equation \eqref{sup} implies that 
%\be\label{2ndlaw}
$\eta_t+\zeta^j_{,j}\ge 0$,
i.e., the second law of thermodynamics holds with
\be\label{EEFnb}
\eta=\rho\varsigma,\quad \zeta^j=\rho \varsigma u^j.
\ee
For solutions to the asscociated homogeneous system, the calculation shows the corresponding identity
\be
\eta_t+\zeta^j_{,j}=0,
\ee  
which means that $(\eta,\zeta^j)$ is an (entropy, entropy flux) pair for it. Assumption \eqref{D2e}
implying that $E$ is a convex function of the conserved quantities $\rho,m^i,S^{kl},\mT,\eta$, we see 
that $-\eta$ is a convex function of the conserved quantities 
$\rho,m^i,E,S^{kl},\mT$.

We have thus proven 
\begin{theo}
Using the (entropy, entropy flux) pair $(\eta,\zeta^j)$,  
system \eqref{YGnb} can be written as a symmetric hyperbolic system of balance laws.
All its solutions generate entropy,
\be\label{2ndlaw}
\eta_t+\zeta^j_{,j}\ge 0.
\ee  
(ii) For any fixed values of $\epsilon >0$ and $\rho_*,\theta_*>0$ and any data 
$({^0\rho},{^0\theta},{^0u}^i,{^0\sigma}^{kl},{^0\tau})$ such that 
$({^0\rho}_0-\rho_*,{^0\theta}-\theta_*,{^0u}^i,{^0\sigma}^{kl},{^0\tau})\in H^s(\R^d), 
s>d/2+1,$ and ${^0\rho},{^0\theta}>0$ are 
uniformly bounded away from $0$,
there exists a $\bar t>0$ such that system \eqref{YGnb} has a solution in $C^1([0,\bar t],H^s(\R^d))$ 
that assumes these data.\\
(iii) If for certain initial data, the nonbarotropic Navier-Stokes equations 
\be\label{NSnb}
\begin{aligned}
\rho_t+(\rho u^j)_{,j}&=0\\
(\rho u^i)_t +(\rho u^ju^i+p\delta^{ij})_{,j}&=
(\mu c^{ijl}_ku^k_{,l}+\nu\delta^{ij}u^k_{,k})_{,j}
\\
\hat E_t+\hat F^j_{,j}&=
((\mu c^{ijl}_ku^k_{,l}+\nu\delta^{ij}u^k_{,k})u_i)_{,j}
\end{aligned}
\ee
with 
\be
\hat E=\rho (e+\frac12 u^2),\quad \hat F^j=(\hat E+p)u^j
\ee
have a solution in $C^1([0,\bar t],H^s(\R^d))$ for some $\bar t>0$,
then there exists an $\bar\epsilon>0$ such that for every $\epsilon\in(0,\bar\epsilon)$ 
system \eqref{YGnb} has a unique 
solution in $C^1([0,\bar t],H^s(\R^d))$ for the same data, augmented via
$\sigma^{kl}=\epsilon\mu c_i^{klj} u^i_{,j}, \tau=\epsilon\nu u^j_{,j}$, 
and the $(\rho,u,s)$-parts of these solutions converge, 
as $\epsilon\searrow 0$, to the solution of \eqref{NSnb}. 
\end{theo}

\textit{Shock waves.} For general fluids with convex internal energy, $e''(v)>0$ or 
$D^2e(v,\varsigma)>0$ as considered
already so far, both the barotropic and the nonbarotropic Navier-Stokes equations support 
shock waves in the sense of heteroclinic traveling wave solutions 
\be\label{RU}
(\rho(t,x),u(t,x))=(R(x\cdot n-ct),U(x\cdot n-ct))
\ee
or
\be\label{RUS}
(\rho(t,x),u(t,x),\theta(t,x))=(R(x\cdot n-ct),U(x\cdot n-ct),\Theta(x\cdot n-ct)),
\ee
respectively \cite{Gi,Sm,Da}, where $n\in S^{d-1}$ is the direction and $c$ the speed
of propagation and $(R,U)$, or $(R,U,\Theta)$ respectively, is the ``profile'' of the shock,
a heteroclinic solution of the associated ODE, with limits 
\be\label{endstates}
(R,U)(\pm\infty)=(\rho^\pm,u^\pm)\!
\text{ or }\!
(R,U,\Theta)(\pm\infty)=(\rho^\pm,u^\pm,\theta^\pm),\ \! 
\text{ respectively,}
\ee
that, as the respective pairs of downstream and upstream states, fulfil the Rankine-Hugoniot jump
conditions.  
The time-asymptotic stability of these viscous shock waves against planar (``longitudinal``, 
``onedimensional'') perturbations has early on been studied in 
\cite{Gm,LiuGen,LiuNSt,SX} and then, Zumbrun et al.\ have shown \cite{ZH,MZ03,MZ04s,MZ04l}
that the nonlinear stability 
reduces completely to the so-called Evans function condition, that is, the requirement that   
the Evans function $D=D^v$, a certain analytic function defined on an open superset of the closed
right half $\bH=\{\lambda:\text{Re}(\lambda)\ge0\}$ of $\C$ (cf., e.\ g., \cite{GZ})
has precisely one zero on $\bH$, namely at $0$,
and this zero is simple. \\
Zumbrun and co-workers have also studied shock profiles in systems of balance 
laws (``relaxed conservation laws``) \cite{MZ02,PZ,MZ05,MZ09}, again with the result that the nonlinear 
stability reduces to the Evans function condition, now applied to $D=D^r$, the relaxation 
counterpart of the viscous Evans function $D^v$.

Now, \eqref{YG} and \eqref{YGnb} are relaxed systems of conservation laws and as such obvious 
candidates for the application of Zumbrun's stability theory. We observe, first in one space dimension:
\begin{theo}
Consider any viscous shock wave solution $(R,U)$ or $(R,U,\Theta)$
%, with \eqref{endstates}, 
of \eqref{NS} or \eqref{NSnb}, respectively. Then the following hold:
(i) For sufficiently small $\epsilon>0$, also the respective relaxed system, i.e., \eqref{YG} or
\eqref{YGnb}, admits a shock wave, $(R_\epsilon,U_\epsilon,\Sigma_\epsilon,T_\epsilon)$ 
or $(R_\epsilon,U_\epsilon,\Theta_\epsilon,\Sigma_\epsilon,T_\epsilon),$ 
with 
$$
(R_\epsilon,U_\epsilon)(\pm\infty)=(R,U)(\pm\infty)
\quad\text{or}\quad
(R_\epsilon,U_\epsilon,\Theta_\epsilon)(\pm\infty)=(R,U,\Theta)(\pm\infty),
\ \ \text{respectively.}
$$
and 
$$
\Sigma_\epsilon(\pm\infty)=T_\epsilon(\pm\infty)=0;
$$ 
these relaxation profiles are regular perturbations of the Navier-Stokes profiles,
$$
\lim_{\eps\to0}(R_\epsilon,U_\epsilon,\Sigma_\epsilon+T_\eps)=(R,U,-{\tilde\mu} U')
\text{ or }
\lim_{\eps\to0}(R_\epsilon,U_\epsilon,\Theta_\epsilon,\Sigma_\epsilon+T_\eps)=(R,U,\Theta,-{\tilde\mu} U'),
\text{respectively.}
$$
(ii) There exist an $\bar\epsilon>0$ and an open superset $\Lambda$ of $\bH$ in $\C$ such that the 
Evans function $D^v$ of the Navier-Stokes shock profile 
and the Evans functions $D^r_\epsilon,0<\epsilon<\bar\epsilon,$ 
of the relaxation profiles are all defined on $\Lambda$ and satisfy there
\be
\lim_{\epsilon\to 0}D^r_\epsilon=D^v,\quad\text{uniformly on compacta.}
\ee
\end{theo}
\begin{proof}
We discuss only the barotropic case (as the nonbarotropic case can be treated in complete analogy). 
Rescaling (and noticing that in 1D, setting $\tilde\mu=(2/3)\mu+\nu$, $\tau$ is redundant) 
we see that system \eqref{YG} reads
\be\label{Y1D}
\begin{aligned}
\rho_t+(\rho u)_x&=0\\
(\rho u)_t+(\rho u^2+p(1/\rho))+\sigma)_x&=0\\
\epsilon((\rho\sigma)_t+(\rho u\sigma)_x)+u_x&=-\frac\sigma{\tilde\mu} . 
\end{aligned}
\ee
(i) Choosing, due to Galilean invariance w.\ l.\ o.\ g., $n=1$ and $c=0$, the ODE governing 
the relaxation shock profiles $(R_\epsilon,U_\epsilon,\Sigma_\epsilon)$ is 
\be\label{relprofileODE}
\begin{aligned}
(RU)'&=0\\
(RU^2+p(1/R)+\Sigma)'&=0\\
\epsilon(RU\Sigma)'+U'&=-\frac\Sigma{\tilde\mu}.
\end{aligned}
\ee
This system is a regular perturbation of its $\epsilon\to 0$ limit 
\be\label{NSprofileODE}
\begin{aligned}
(RU)'&=0\\
(RU^2+p(1/R)+\Sigma)'&=0\\
U'&=-\frac\Sigma{\tilde\mu},
\end{aligned}
\ee
which has the same rest points. 
But \eqref{NSprofileODE} is the system governing the Navier-Stokes shock profiles.
In both ODEs, $\Sigma$ vanishes at the rest points; in the latter one, $\Sigma=-{\tilde\mu} U'$ identically
along the profile. Going up from \eqref{NSprofileODE} to \eqref{relprofileODE}, any heteroclinic
solution persists uniquely for small $\epsilon>0$, as for $\epsilon=0$ its orbit is a 
\emph{transverse} intersection
of the unstable manifold of one with the stable manifold of the other hyperbolic fixed point 
(cf., e.g., \cite{GZ,ZS,FS1} for the idea). 

(ii) The Evans function for a shock profile $(R,U,\Sigma)$ of the relaxed system \eqref{Y1D}
is a particular Wronskian of solutions of the eigenvalue problem for the linearization 
of \eqref{Y1D} around the profile,
\be\label{EVP}
\begin{aligned}
\lambda r+Ur'+R\upsilon'&=0\\
\lambda(Ur+R\upsilon)+U^2r'+2m\upsilon'-R^{-2}p'(1/R)r'+s'&=0\\
\epsilon\lambda Rs+\epsilon ms'+\upsilon'&=\frac{-s}{\tilde\mu};
\end{aligned}
\ee 
in \eqref{EVP}, $m$($\neq 0$) is the constant value that $RU$ takes along the profile according to the first 
line of \eqref{NSprofileODE}; note that the mass transfer $m$ is $\neq 0$ for any genuine shock, and 
thus $U\neq0$ uniformly along the shock profile. System \eqref{EVP} again is a 
regular perturbation of its $\epsilon\to 0$ limit 
\be\label{EVP0}
\begin{aligned}
\lambda r+Ur'+R\upsilon'&=0\\
\lambda(Ur+R\upsilon)+U^2r'+2m\upsilon'-R^{-2}p'(1/R)r'+s'&=0\\
\upsilon'&=\frac{-s}{\tilde\mu},
\end{aligned}
\ee 
and this immediately implies that for any $C>0$, there exist $\delta(C),\bar\epsilon >0$ such that 
the Evans functions $D^v$  and $D^r_\epsilon, 0<\epsilon<\bar\epsilon,$ 
of \eqref{EVP0} and  \eqref{EVP} respectively, are well-defined 
on the domain 
$$\Lambda_C=\{\lambda\in \C:|\lambda|\le C
\text{ and Re}(\lambda) \ge -\delta(C)\}$$   
and satisfy $\lim_{\epsilon\to 0}D^r_\epsilon=D^r$ uniformly on $\Lambda_C$. If now $D^v$ has only one, simple,
zero on $\Lambda_C$, the same thus obviously holds for $D^r_\epsilon$ as soon as 
$\epsilon$ is sufficiently small.
(That this zero is $0$ follows from the general fact that $0$ always is an eigenvalue (namely
the one corresponding to shifting the profile; cf., e.g., \cite{GZ}) or \cite{FS1}.) 
Finally, the standard arguments (cf., e.g., again \cite{GZ}) showing the absence 
of Evans function zeroes outside a circle  $|\lambda|\ge C$ of sufficiently large radius $C$ here 
trivially apply \emph{uniformly} to \eqref{EVP0} and \eqref{EVP}, so that altogether, 
for sufficiently small $\epsilon>0$, 
$D^r_\epsilon$ satisfies the complete Evans function condition when $D^v$ does. 
\end{proof}
\vspace{-.5cm} 
The cited works of Zumbrun and coworkers now readily imply the following. 
\begin{coro}
For any Navier-Stokes shock that satisfies Majda's condition on the  
Lopatinski determinant \cite{Ma} and Zumbrun's Evans function condition,
all corresponding relaxation shocks as identified in Assertion (i) of Theorem 3
with sufficently small $\epsilon>0$ are nonlinearly stable with respect to 
Sobolev norms chosen appropriately according to Zumbrun theory. 
\end{coro}
\vspace{-.5cm}
Note that the corollary holds also regarding non-planar perturbations, i.e., as a fully
multidimensional result. This is obvious from the fact that the above reasoning 
carries over verbatim to the multidimensional Evans function $D=D(\lambda,\omega)$ whose 
second argument is a transverse Fourier frequency (cf.\ \cite{Z,ZS,FS2}). 

The use of the corollary lies in the fact the two said conditions have been verified 
for large families of Navier-Stokes shocks, of small and large amplitude, cf.\ 
\cite{HZ02,BHRZ08,HLZ09,HLZ17}. ---

\textit{Ruggeri's model.}
Models like the above belong to the realm of rational extended thermodynamics \cite{MR}. 
Exactly that model was indeed proposed almost 40 years ago by Ruggeri~\cite{R}.
\begin{defi}\label{GV4P}
(i) One calls $Y=(Y_0,\ldots,Y_n)$ \emph{Godunov variables}
and a vector $(X^\alpha(Y))_{\alpha=0,....,3}$ a 
\emph{4-potential} for a symmetric system of conservation laws
\be\label{claws}
\partial_\alpha F^{\alpha b}(Y)=I^{\alpha b}(Y),\quad b=0,\ldots,n,
\ee
if the spatiotemporal fluxes $F^{\alpha b}$ derive from $X^\alpha$ as  
\be\label{Fis}
F^{\alpha b}(Y)=\frac{\partial X^\alpha(Y)}{\partial Y_b}, \quad\alpha=0,...,3,\ b=0,...,n. 
\ee
(ii) In that case, system \eqref{claws} is called \emph{symmetric-hyperbolic} if 
\be\label{Hessianpositive}
\left(\frac{\partial^2 X^0(Y)}{\partial Y_b\partial Y_c}\right)_{b,c=0,\ldots,n}
\quad \text{is definite.} 
\ee
(iii) In the situation of (i), a scalar function $X(Y)$ is called a \emph{protopotential} for 
system \eqref{claws} if 
$$
X^\alpha(Y)=\frac{\partial X(Y)}{\partial Y_ \alpha}.
$$
(iv) If a system as in (i) is Lorentz invariant, it is called \emph{causal} if  
\be\label{Hessianpositive}
\left(\frac{\partial^2 X^\alpha(Y)}{\partial Y_b\partial Y_c}T_\alpha\right)_{b,c=0,\ldots,n}
\quad \text{is definite for all }T_\alpha \text{ with }T_\alpha T^\alpha<0. 
\ee
\end{defi}
\newpage
For the Euler equations of classical inviscid fluid dynamics, the Godunov variables 
\be\label{GVnonrelEuler}
\tilde\psi\equiv\psi-\frac{u^2}{2\theta},
\ \tilde u^i\equiv\frac {u^i}\theta,
\ \tilde\theta\equiv\frac 1\theta
\ee
(where $\psi=g/\theta$ with $g=e+pv-\theta\varsigma$ the chemical potential)
and the corresponding 4-potential 
\be\label{GPnonrelEuler}
X_E^0=\frac p\theta,\quad X_E^i=\frac p\theta u^i   
\ee
were proposed by Godunov \cite{Go} (while the production term is $I_E^{\alpha b}=0$
in that case).
In \cite{R}, Ruggeri augmented the list \eqref{GVnonrelEuler} as
\be\label{GVnonrelExtended} 
\tilde\psi,\ \tilde u^i,\ \tilde\theta,
\ \tilde\sigma^{ij}\equiv\frac{\sigma^{ij}}\theta,
\ \tilde\tau\equiv\frac\tau\theta,  
\ \tilde q^i \equiv\frac{q^i}{\theta^2},
\ee
and showed two things:
(a) With Godunov variables \eqref{GVnonrelExtended}, the 4-potential
\be\label{GPnonrelNSF}
X^0=\frac p\theta,
\quad 
X^i=\frac1\theta\left(\left(p\delta^{ij}-(\sigma^{ij}+\tau \delta^{ij)}\right)u_j+q^i\right),
\ee
and the production term
\be
(I^{b})_{b=0,...,13}=(0_1,0_3,0_1,-\frac1{\theta\eta_Q} q,\frac1{\theta\eta_S}\sigma,\frac1{\theta\eta_B}\tau),
\ee
system \eqref{claws}, \eqref{Fis} coincides with the classical Navier-Stokes-Fourier (NSF) 
equations. In \eqref{GPnonrelEuler} and \eqref{GPnonrelNSF} the fluid can be identified through
an equation of state   
\be\label{eos}
p=\hat p(\theta,\psi).  
\ee
(b) With the modified 4-potential 
\be\label{GPNew}
\begin{aligned}
X_N^0=\frac{p_N}\theta,
\quad 
X_N^i=\frac1\theta\left(\left(p_N\delta^{ij}-(\sigma^{ij}+\tau \delta^{ij)}\right)u_j+q^i\right),
\qquad\quad\\
p_N=\hat p(\theta,\psi-\mathcal S),\quad
\mathcal S=\mathcal  S(\tilde\sigma,\tilde\tau,\tilde q)
=\frac12\left(\eps_S\Vert\tilde\sigma\Vert^2+\eps_B\tilde\tau^2+\eps_Q|\tilde q|^2\right),
\end{aligned}
\ee
and the unmodified production term $I^{\alpha b}_N=I^{\alpha b}$,
the resulting system is, as he called it, ``a symmetric-hyperbolic system of conservative equations 
for a viscous heat conducting fluid''. The latter is indeed the model discussed above!

A 4-potential comes with a protopotential whenever the rank-2 tensor
$$
\left(F^{\alpha\beta}\right)_{\alpha,\beta=0,...,3}
$$  
is symmetric. This has been used by Geroch and Lindblom in the relativistic context with 
$F^{\alpha\beta}=T^{\alpha\beta}$ the Lorentz invariant energy-momentum 4-tensor \cite{GL}; 
it also applies to non-relativistic fluid dynamics with $F^{\alpha\beta}=M^{\alpha\beta}$ 
the Galilei invariant mass-momentum 4-tensor.
We note that the classical Euler equations, Navier-Stokes-Fourier equations, and Ruggeri 
model have 
$$
\begin{aligned}
X_E(\tilde\psi,\tilde u,\tilde\theta)&=\hat X(\theta,\psi)\\
X(\tilde\psi,\tilde u,\tilde\theta,\tilde\sigma,\tilde\tau,\tilde q)
&=\hat X(\theta,\psi)
+
\theta\left(-(1/2)(\tilde\sigma_{ij}+\tilde\tau\delta_{ij})\tilde u^i\tilde u^j
+\tilde q_i\tilde u^i\right)\\
X_N(\tilde\psi,\tilde u,\tilde\theta,\tilde\sigma,\tilde\tau,\tilde q)&=
\hat X(\theta,\psi-\mathcal S(\tilde\sigma,\tilde\tau,\tilde q))
+
\theta\left(-(1/2)(\tilde\sigma_{ij}+\tilde\tau\delta_{ij})\tilde u^i\tilde u^j
+\tilde q_i\tilde u^i\right)\
\end{aligned}
$$
as protopotentials, respectively, where 
$\hat X$ relates to \eqref{eos} as 
$$
\hat X(\theta,\psi)=\int\hat p(\theta,\psi)d\psi.
$$
From the original perspective of the present paper we switch to a new one, now proposing a

\textit{Lorentz invariant version of Ruggeri's model.}  
The relativistic Euler equations,
$$
\begin{aligned}
 \partial_\alpha T^{\alpha\beta}_E&=0,\\
 \partial_\alpha N^\alpha_E&=0
\end{aligned}
$$
admit the Godunov variables \cite{RS} 
\be\label{GVrelEuler}
\psi=\frac g\theta,\quad\Upsilon^\alpha=\frac{u^\alpha}\theta
\ee
and the protopotential \cite{GL} 
\be\label{XrelEuler}
\bar X_E(\Upsilon,\psi)=\bar{\hat X}(\theta,\psi)
\ee
with now 
$$
\bar{\hat X}(\theta,\psi)=\int\hat p(\theta,\psi)d\theta;
$$
indeed, with $\theta^{-2}=-\Upsilon_\alpha\Upsilon^\alpha$,
\be\label{idealgas}
\begin{aligned}
T^{\alpha\beta}_E&=
\frac{\partial^2\bar X_E(\Upsilon,\psi)}{\partial \Upsilon_\alpha\partial\Upsilon_\beta}=
\theta^3\frac{\partial p(\theta,\psi)}{\partial \theta}\Upsilon^\alpha\Upsilon^\beta
+p(\theta,\psi)g^{\alpha\beta},\\
N^\alpha_E&=
\frac{\partial^2\bar X_E(\Upsilon,\psi)}{\partial \Upsilon_\alpha\partial\psi}
=\frac{\partial p(\theta,\psi)}{\partial \psi}\Upsilon^\alpha.
\end{aligned}
\ee
To prepare for including dissipative effects, we decompose symmetric rank-2 tensors via
$$
\begin{aligned}
\SP:V_{\alpha\beta}\mapsto 
\SP_{\alpha\beta}^{\gamma\delta}V_{\gamma\delta}&\quad\text{with }
\SP_{\alpha\beta}^{\gamma\delta}=
\Pi_{(\alpha}^\gamma\Pi_{\beta)}^\delta-\frac13\Pi_{\alpha\beta}\Pi^{\gamma\delta},
\\
\SP:V_{\alpha\beta}\mapsto\BP_{\alpha\beta}^{\gamma\delta}V_{\gamma\delta}&\quad\text{with }
\BP_{\alpha\beta}^{\gamma\delta}=
\frac13\Pi_{\alpha\beta}\Pi^{\gamma\delta},
\\
\QP:V_{\alpha\beta}\mapsto\QP_{\alpha\beta}^{\gamma\delta}V_{\gamma\delta}&\quad\text{with }
\QP_{\alpha\beta}^{\gamma\delta}=
-\left(\Pi_\alpha^{(\gamma} u_\beta+\Pi_\beta^{(\gamma} u_\alpha\right) u^{\delta)},
\\
\HP:V_{\alpha\beta}\mapsto\HP_{\alpha\beta}^{\gamma\delta}V_{\gamma\delta}&\quad\text{with }
\HP_{\alpha\beta}^{\gamma\delta}=
u_\alpha u_\beta u^\gamma u^\delta;
\end{aligned}
$$  
the mappings $\SP,\BP,\QP,\HP$ are complementing projectors, i.e., 
$$
{\CP}\DP=\delta^{CD}\ \CP,\quad C,D\in\{S,B,Q,H\},
\quad \text{and}\quad 
\SP+\BP+\QP+\HP=\text{id}.
$$
\goodbreak
Now, slightly varying upon \cite{GL}, we augment, as a first step, the list \eqref{GVrelEuler}
as 
\be\label{GVrelExtended}
\psi,\quad\Upsilon^\alpha,\quad\Sigma^{\alpha\beta}  
\ee
with a single symmetric rank-2 tensor $\Sigma^{\alpha\beta}$, and find that 
\be\label{XEckart}
\bar X(\Upsilon,\psi,\Sigma)=\bar{\hat X}(\theta,\psi)-\frac12\Sigma_{\alpha\beta}\Upsilon^\alpha\Upsilon^\beta
\ee
as a protopotential and the production term
$
(\bar I^b)_b=(0_5,(I^{\alpha\beta})_{\alpha\beta})
$
with
$$
I_{\alpha\beta}=
-\frac1\theta
\left(
 \frac1{\eta_S} \SP_{\alpha\beta}^{\gamma\delta}
-\frac1{\eta_B} \BP_{\alpha\beta}^{\gamma\delta}
-\frac1{\eta_Q} \QP_{\alpha\beta}^{\gamma\delta}
-\frac1{\eta_H} \HP_{\alpha\beta}^{\gamma\delta}
\right)
\Sigma_{\gamma\delta}
$$
together induce Eckart's counterpart \cite{E}
of the non-relativistic NSF equations, here in the form
\be\label{relNSF}
\begin{aligned}
\partial_\alpha \left(T^{\alpha\beta}_E-\Sigma^{\alpha\beta}\right)&=0\\
\partial_\alpha N^\alpha_E&=0\\
\partial_\alpha 
\left(\frac12\left(g^{\alpha\beta}\Upsilon^\gamma+g^{\alpha\gamma}\Upsilon^\beta\right) \right) 
&=I^{\beta\gamma}.
\\
\end{aligned}
\ee
The last line of \eqref{relNSF} is to be read as 
\be\label{decompSigma}
\Sigma_{\alpha\beta}=
-\theta\left(
 {\eta_S} \SP_{\alpha\beta}^{\gamma\delta}
+{\eta_B} \BP_{\alpha\beta}^{\gamma\delta}
+{\eta_Q} \QP_{\alpha\beta}^{\gamma\delta}
+{\eta_H} \HP_{\alpha\beta}^{\gamma\delta}
\right)
\left(\Upsilon_{\gamma,\delta}+\Upsilon_{\delta,\gamma}\right)/2,
\ee
which essentially\footnote{with the exception of the last term, 
which could be there  mathematically (if $\eta_H\neq 0$), but 
to the author's knowledge has not been reported in the physics literature} expresses the classically known 
composition of dissipative stresses (cf., e.g., \cite{W}, p.\ 56),
with $\eta_S,\eta_B,\eta_Q$ coefficients of shear viscosity, bulk viscosity, and heat conduction.
Representation \eqref{decompSigma} and equations \eqref{relNSF} remain meaningful if one 
or several of the dissipation coefficients $\eta_S,\eta_B,\eta_Q,\eta_H$ vanish 
identically;\footnote{While it seems generally reasonable to assume that the `heating dissipation' 
coefficient $\eta_H$ vanishes, many fluids also have $\eta_B=0$, no bulk viscosity.} 
mathematically such situations imply a reduction of the state space to a corresponding equilibrium 
manifold. \\
We finally vary the protopotential $\bar X$ as
\be\label{relRuggeri}
\bar X_N(\Upsilon,\psi,\Sigma)
=
\bar{\hat X}(\theta,\psi-\bar{\mathcal S}(\Sigma,\Upsilon))
-\frac12\Sigma_{\alpha\beta}\Upsilon^\alpha\Upsilon^\beta
\ee
with
$$
\bar{\mathcal S}(\Sigma,\Upsilon)=
\frac12\left(
 \eps_S \SP_{\alpha\beta}^{\gamma\delta}{\Sigma}^{\alpha\beta}{\Sigma}_{\gamma\delta}   
+\eps_B \BP_{\alpha\beta}^{\gamma\delta}{\Sigma}^{\alpha\beta}{\Sigma}_{\gamma\delta} 
+\eps_Q \QP_{\alpha\beta}^{\gamma\delta}{\Sigma}^{\alpha\beta}{\Sigma}_{\gamma\delta}   
+\eps_H \HP_{\alpha\beta}^{\gamma\delta}{\Sigma}^{\alpha\beta}{\Sigma}_{\gamma\delta}
\right).
$$
It appears that the resulting symmetric system
\be
\begin{aligned}
\partial_\alpha\left(\frac{\partial \bar X_N(\Upsilon,\psi,\Sigma)}{\partial \Upsilon_\alpha\partial \Upsilon_\beta}\right)&=0,\\
\partial_\alpha\left(\frac{\partial \bar X_N(\Upsilon,\psi,\Sigma)}{\partial \Upsilon_\alpha\partial \psi}\right)&=0,\\
\partial_\alpha\left(\frac{\partial \bar X_N(\Upsilon,\psi,\Sigma)}{\partial \Upsilon_\alpha\partial \Sigma_{\beta\gamma}}\right)&=0
\end{aligned}
\ee
is well-behaved.
\goodbreak
\end{document}